\documentclass[12pt, reqno]{amsart}
\usepackage{amsmath, amsthm, amscd, amsfonts, amssymb, graphicx}
\usepackage[bookmarksnumbered, colorlinks, plainpages]{hyperref}
\usepackage{mathrsfs}

\DeclareMathAlphabet\EuScript{U}{eus}{m}{n}
\DeclareMathAlphabet\EuScriptBold{U}{eus}{b}{n}
\DeclareMathAlphabet\Eurm{U}{eur}{m}{n}
\DeclareMathAlphabet\Eurb{U}{eur}{b}{n}
\newcommand{\mathcalb}{\EuScriptBold}

\newtheorem{theorem}{Theorem}[section]
\newtheorem{lemma}[theorem]{Lemma}

\theoremstyle{definition}
\newtheorem{definition}[theorem]{Definition}

\theoremstyle{remark}
\newtheorem{remark}[theorem]{Remark}

\numberwithin{equation}{section}

\begin{document}

\newcommand\supess{\mathop{\rm sup\,ess}}
\newcommand\diam{\mathop{\rm diam}}
\newcommand{\AAA}{ {\mathscr A} }
\newcommand{\BBB}{ {\mathscr B} }
\newcommand{\CCC}{ {\mathscr C} }
\newcommand{\DDD}{ {\mathscr D} }
\newcommand{\BBBog}{ {\mathcalb B} }
\newcommand{\WWW}{\Omega}
\newcommand{\WWWalg}{\mathfrak M}
\newcommand{\iii}{\infty}
\newcommand{\ttt}{\rightarrow}
\newcommand{\xx}{{}^\ast\!\otimes}
\newcommand{\ox}{{}^\ast\!\otimes}
\renewcommand{\lll}{\left}
\newcommand{\rrr}{\right}
\newcommand{\llm}{\left(}
\newcommand{\rrm}{\right)}
\newcommand{\lls}{\left[}
\newcommand{\rrs}{\right]}
\newcommand{\llv}{\left\{}
\newcommand{\rrv}{\right\}}
\newcommand{\llp}{\left<}
\newcommand{\rrp}{\right>}
\newcommand{\lla}{\left|}
\newcommand{\rra}{\right| }
\newcommand{\lln}{\left\Vert}
\newcommand{\rrn}{\right\Vert}
\newcommand{\llu}{\left|\!\left|\!\left|}
\newcommand{\rru}{\right|\!\right|\!\right|}
\newcommand{\lle}{\lefteqn}
 \newcommand{\llao}{\vert}
 \newcommand{\rrao}{\vert}
 \newcommand{\llaj}{\bigl\vert}
 \newcommand{\rraj}{\bigr\vert}
 \newcommand{\llad}{\Bigl\vert}
 \newcommand{\rrad}{\Bigr\vert}
 \newcommand{\llat}{\biggl\vert}
 \newcommand{\rrat}{\biggr\vert}
 \newcommand{\llac}{\Biggl\vert}
 \newcommand{\rrac}{\Biggr\vert}
 \newcommand{\zza}[1]{\left\vert{#1}\right\vert}   
 \newcommand{\zzao}[1]{\vert{#1}\vert}
 \newcommand{\zzaj}[1]{\bigl\vert{#1}\bigr\vert}
 \newcommand{\zzad}[1]{\Bigl\vert{#1}\Bigr\vert}
 \newcommand{\zzat}[1]{\biggl\vert{#1}\biggr\vert}
 \newcommand{\zzac}[1]{\Biggl\vert{#1}\Biggr\vert}
 \newcommand{\zzm}[1]{\left({#1}\right)}           
 \newcommand{\llno}{\vert\vert}
 \newcommand{\rrno}{\vert\vert}
 \newcommand{\llnj}{\bigl\vert\!\bigl\vert}
 \newcommand{\rrnj}{\bigr\vert\!\bigr\vert}
 \newcommand{\llnd}{\Bigl\vert\!\Bigl\vert}
 \newcommand{\rrnd}{\Bigr\vert\!\Bigr\vert}
 \newcommand{\llnt}{\biggl\vert\!\biggl\vert}
 \newcommand{\rrnt}{\biggr\vert\!\biggr\vert}
 \newcommand{\llnc}{\Biggl\vert\!\Biggl\vert}
 \newcommand{\rrnc}{\Biggr\vert\!\Biggr\vert}
 \newcommand{\lluu}{\vert\vert\vert}
 \newcommand{\rruu}{\vert\vert\vert}
 \newcommand{\lluo}{\vert\!\:\!\vert\!\:\!\vert}
 \newcommand{\rruo}{\vert\!\:\!\vert\!\:\!\vert}
 \newcommand{\lluj}{\bigl\vert\!\bigl\vert\!\bigl\vert}
 \newcommand{\rruj}{\bigr\vert\!\bigr\vert\!\bigr\vert}
 \newcommand{\llud}{\Bigl\vert\!\Bigl\vert\!\Bigl\vert}
 \newcommand{\rrud}{\Bigr\vert\!\Bigr\vert\!\Bigr\vert}
 \newcommand{\llut}{\biggl\vert\!\biggl\vert\!\biggl\vert}
 \newcommand{\rrut}{\biggr\vert\!\biggr\vert\!\biggr\vert}
 \newcommand{\lluc}{\Biggl\vert\!\Biggl\vert\!\Biggl\vert}
 \newcommand{\rruc}{\Biggr\vert\!\Biggr\vert\!\Biggr\vert}
 \newcommand{\zzn}[1]{\left|\left|{#1}\right|\right|}           
 \newcommand{\zzu}[1]{\zza\!{\zza{\!\zza{#1}\!}\!}}         
 \newcommand{\zzuo}[1]{|\!\:\!|\!\:\!|{#1}|\!\:\!|\!\:\!|}
 \newcommand{\zzuj}[1]{\zzaj{\!\zzaj{\!\zzaj{#1}\!}\!}}
 \newcommand{\zzud}[1]{\zzad{\!\zzad{\!\zzad{#1}\!}\!}}
 \newcommand{\zzut}[1]{\zzat{\!\zzat{\!\zzat{#1}\!}\!}}
 \newcommand{\zzuc}[1]{\zzac{\!\zzac{\!\zzac{#1}\!}\!}}
\font\cyr=wncyr10
\let\CMcal=\mathcal
\newcommand{\ggint}{{\textrm{\rm\scriptsize G}}\!\!\!\;\!\!\!\int}
\newcommand{\gggint}{{\textrm{\rm\scriptsize G}}\!\!\!\!\!\int}
\newcommand\ipt{ i.p.t. integral }
\newcommand\wmeasur{{$[\mu]$ weakly$^*$-mea\-surable }}
\newcommand{\bbproo}{ {\it Proof:\/} }
\newcommand{\eeproo}{$ $\hfill $ $\qed}
\newcommand{\bbdef}{\begin{definition}\sl }
\newcommand{\eedef}{\end{definition}\rm }
\newcommand{\bbl}{\begin{lemma}}
\newcommand{\eel}{\end{lemma}}
\newtheorem{corollary}[theorem]{Corollary}
\newcommand{\proo}{ {\it Proof:\/} }
\newcommand{\sui}{\sum_{i=1}^I}
\newcommand{\limm}{\lim_{n\ttt\infty}}
\newcommand{\suji}{\sum_{j=1}^\iii}
\newcommand{\suni}{\sum_{n=1}^\iii}
\newcommand{\tr}{\textrm{\rm tr}}
\newcommand{\AAt}{ {\mathscr{A}}_t }
\newcommand{\AAs}{ {\mathscr{A}}_s }
\newcommand{\BBt}{ {\mathscr{B}}_t }
\newcommand{\HH}{{\mathcal H}}
\newcommand{\BH}{ {\mathcalb B}({\mathcal H})}
\newcommand{\ccc}{ {\mathcalb C}}
\newcommand{\ccj}{ {\mathcalb  C}_1({\CMcal H}) }
\newcommand{\cci}{ {\mathcalb  C}_\iii({\CMcal H}) }
\newcommand{\ccu}{ {\mathcalb  C}_{\lluo\cdot\rruo}({\CMcal H}) }
\newcommand{\cco}{ {\mathcalb  C}_{\llo\cdot\rro}({\CMcal H}) }
\newcommand{\dm}{\,d\mu}
\newcommand{\dt}{\,d\mu(t)}
\newcommand{\LJ}{ L^1(\WWW,\mu)}
\newcommand{\LD}{{ L^2(\WWW,\mu)}}
\newcommand{\LDH}{ {L^2(\Omega,\mu,{\mathcal H})}}
\newcommand{\LDBH}{{ L_G^2(\Omega,\mu,\BH)}}
\newcommand{\LDCU}{{ L_G^2(\Omega,\mu,\ccu)}}
\newcommand{\LI}{ L^\iii(\WWW,\mu)}
\newcommand{\LP}{ L^p(\WWW,\mu)}
\newcommand{\inN}{{\in\mathbb N}}
\newcommand{\inR}{{\in\mathbb R}}
\newcommand{\inC}{{\in\mathbb C}}
\newcommand{\gint}{{\textrm{\rm\scriptsize G}}\!\!\!\;\!\!\!\int}
\newcommand{\Gint}{{\textrm{\rm\ G}}\!\!\!\!\!\!\int\nolimits}
\newcommand{\innE}{\int\nolimits_{E}}
\newcommand{\InE}{{\CMcal I}_{E}}
\newcommand{\JnE}{{\CMcal J}_{E}}
\newcommand{\Inaf}{{\CMcal I}_f{}}
\newcommand{\Jnag}{{\CMcal J}_g{}}
\newcommand{\LL}{L^1(\WWW,d\mu)}
\newcommand{\mmm} {$[\mu]$ a.e. on $\WWW$ }

\title[Landau and Gr\" uss type  inequalities]{Landau and Gr\" uss type inequalities for inner product type integral transformers
in norm ideals}

\author{Danko R. Joci\'c}
\address{University of Belgrade, Department of  Mathematics, Studentski trg 16,
P.O.box 550, 11000 Belgrade, Serbia}
\email{jocic@matf.bg.ac.rs}


\newcommand{\Dj}{\mbox{-\kern-.4em D}}

\author{\Dj OR\Dj E KRTINI\'C}
\address{University of Belgrade, Department of  Mathematics, Studentski trg 16,
11000 Belgrade, Serbia} \email{georg@matf.bg.ac.rs}

\author{Mohammad Sal Moslehian}
\address{Department of Pure Mathematics, Center of Excellence in Analysis on Algebraic Structures
(CEAAS), Ferdowsi University of Mashhad, P.O. Box 1159, Mashhad
91775, Iran.\\ \url{http://profsite.um.ac.ir/~moslehian/}}
\email{moslehian@ferdowsi.um.ac.ir}\email{moslehian@ams.org}





\begin{abstract}
For a probability measure $\mu$ and for square integrable fields
$(\mathscr{A}_t)$ and $(\mathscr{B}_t)$ ($t\in\Omega$) of commuting
normal operators we prove Landau type inequality
\begin{multline*}
\llu\int_\Omega\mathscr{A}_tX\mathscr{B}_td\mu(t)-
      \int_\Omega\mathscr{A}_t\,d\mu(t)X\!\!\int_\Omega\mathscr{B}_t\,d\mu(t)
\rru \\
\le \llu
\sqrt{\,\int_\Omega|\mathscr{A}_t|^2\dt-\left|\int_\Omega\mathscr{A}_t\dt\right|^2}X
\sqrt{\,\int_\Omega|\mathscr{B}_t|^2
\dt-\left|\int_\Omega\mathscr{B}_t\dt\right|^2} \rru
\end{multline*}
for all  $X\in\mathcalb{B}(\mathcal{H})$ and for all unitarily
invariant norms $\lluo\cdot\rruo$.

For  Schatten $p$-norms similar inequalities are given for arbitrary
double square integrable fields. Also, for all bounded self-adjoint
fields satisfying $C\le\mathscr{A}_t\le D$  and
$E\le\mathscr{B}_t\le F$  for all $t\in\Omega $ and some  bounded
self-adjoint operators $C,D,E$ and $F$, then for all $X\in\ccu$ we
prove Gr\"uss type inequality
\begin{equation}
\llu\int_\Omega\mathscr{A}_tX\mathscr{B}_t \dt-     \int_\Omega
\mathscr{A}_t\,d\mu(t)X\!\!   \int_\Omega\mathscr{B}_t\,d\mu(t)
\rru\leq \frac{\|D-C\|\cdot\|F-E\|}4\cdot\lluo X\rruo.\nonumber
\end{equation}
More general results for arbitrary bounded fields are also given.
\end{abstract}

\maketitle

\section{Introduction}

The Gr\" uss inequality \cite{GRU}, as a complement of Chebyshev's
inequality, states that if $f$ and $g$ are integrable real functions
on $[a,b]$ such that $C\leq f(x)\le D$ and $E\leq g(x)\le F$ hold
for some  real constants $C,D,E,F$ and for all $x\in[a,b]$, then
\begin{equation}
  \left|\frac1{b-a}\int_a^bf(x)g(x)dx-\frac1{(b-a)^2}\int_a^b f(x)dx\int_a^b g(x)dx\right|
\leq\frac14(D-C)(F-E)\,;\label{grisovaca}
\end{equation}
see \cite{M-M} for several proofs of this inequality in the discrete
form. It has been the subject of intensive investigation, in which
conditions on functions are varied to obtain different estimates;
see \cite{DRA2,M-P-F} and references therein. This inequality has
been investigated, applied and generalized by many authors  in
different areas of mathematics, among others  in inner product
spaces \cite{DRA1}, quadrature formulae \cite{CHE,UJE}, finite
Fourier transforms \cite{B-C-D-R}, linear functionals
\cite{A-B,I-P-T}, matrix traces \cite{REN}, inner product modules
over $H^*$-algebras and $C^*$-algebras \cite{B-I-V, I-V}, positive
maps \cite{M-R} and completely bounded maps \cite{P-R}.

\subsection{Symmetric gauge functions, unitarily invariant norms
            and their norm ideals}

Let $\mathcalb{B}(\mathcal{H})$ and
$\mathcalb{C}_\infty(\mathcal{H})$ denote respectively spaces of all
bounded and all compact linear operators acting on a separable,
complex Hilbert space $\mathcal{H}$. Each "symmetric gauge" (s.g.)
function $\Phi$ on sequences gives rise to  a unitarily invariant
(u.i) norm on operators defined by
$\left\|X\right\|_\Phi=\Phi(\{s_n(X)\}_{n=1}^\infty)$, with
$s_1(X)\ge s_2(X)\ge\hdots$ being the singular values of $X$, i.e.,
the eigenvalues of $|X|=(X^*X)^\frac12.$ We will denote by the
symbol $\left|\left|\left|\cdot\right|\right|\right|$ any such norm,
which is therefore defined on a naturally associated norm ideal
$\ccu$ of $\mathcalb{C}_\infty(\mathcal{H})$ and  satisfies the
invariance property $\lluo UXV\rruo=\lluo X\rruo$ for all $X\in\ccu$
and for all unitary operators $U,V\in\BH$.

Specially well known among u.i.  norms are the Schatten $p$-norms
defined for $1\le p<\infty$ as $\|X\|_p=\sqrt[p]{\,\sum_{n=1}^\infty
s_n^p(X)}$, while $\|X\|_\infty =\|X\|=s_1(X)$ coincides with the
operator  norm $\|X\|$. Minimal and maximal u.i. norm are among
Schatten norms, i.e., $\|X\|_\infty\le\llu X\rru\le\|X\|_1$ for all
$X\in\mathcalb{C}_1(\mathcal{H})$ (see inequality (IV.38) in
\cite{bhatijinaknjiga}). For $f,g\in\mathcal{H}$, we will denote by
$g^*\otimes f$ one dimensional operator $(g^*\otimes f)h=\langle
h,g\rangle f$ for all $h\in\mathcal{H}$, known that the linear span
of $\{g^*\otimes f\,|\, f,g\in\HH$ is dense in each of
$\mathcalb{C}_p(\mathcal{H})$ for $1\le p\le\infty$. Schatten
$p$-norms are also classical examples of $p$-reconvexized norms.
Namely, any u.i. norm  $\lln\cdot\rrn_\Phi$ could be
$p$-reconvexized for any $p\ge1$ by setting $\lln A\rrn_{\Phi^{(p)}}
= \lln |A|^p\rrn_{\Phi}^{\frac1p}$ for all $A\in\BH$ such that
$|A|^p\in \ccc_{\Phi}(\HH)$. For the proof of the triangle
inequality and other properties of these norms see preliminary
section in \cite{joc09ii}; for the characterization of the dual norm
for $p$-reconvexized one see Th. 2.1 in \cite{joc09ii}.

\subsection{Gel'fand integral of operator valued functions}

Here we will recall the basic properties and terminology related to
the notion of Gel'fand integral, when it applies to operator valued
(o.v.) functions. Since this theory is well known, we give those
properties without the proof.
Following \cite{du}, p.53., if $(\WWW,\WWWalg,\mu)$ is a measure
space,
the mapping $\AAA:\WWW\ttt \BH$ will be called \wmeasur if  a scalar
valued function $t\ttt \tr (\AAt Y)$ is measurable for any
$Y\in\ccj$. In addition, if  all these functions are in $\LJ$, then
according to the fact that $\BH$ is the dual space of $\ccj$, for
any $E\in\WWWalg$  there will be the unique operator $\InE\in\BH$,
called the Gel'fand ({\cyr Gelp1fand}) or weak $^*$-integral of
$\AAA$ over $E$, such that
\begin{equation}
\tr(\InE Y)=\int_E \tr(\AAt Y)\dt \textrm{\qquad for all
$Y\in\ccj$.} \label{geljfandovintegral}
\end{equation}
We will denote it  by $\int_E\AAt\dt$, $\int_E\AAA\dm$ or
exceptionally by
                      $\gggint_{\,E}\AAA\dm,$
if the context requires to distinguish this one from other types of
integration.

A  practical tool for this type of integrability to deal with is the
following
\begin{lemma}\label{lema1}

 $\AAA:\WWW\ttt \BH$ is \wmeasur
 (resp. $[\mu]$ weakly $^*$-integrable) iff
 scalar valued functions $t\ttt \left<\AAt f,f\right>$ are  $[\mu]$ measurable (resp. integrable) for every
$f\in\HH$.

\end{lemma}

In view of Lemma \ref{lema1},
the basic definition (\ref{geljfandovintegral}) of Gel'fand integral
for o.v. functions can be reformulated as follows:
\begin{lemma}\label{novadefinicionalema}

If  $\left<\AAA f,f\right>\in L^1(E,\mu)$ for all $f\in\HH$, for
some $E\in\WWWalg$ and a $\BH$-valued function $\AAA$ on $E$, then
the mapping $f\ttt\int_E  \left<\AAt f,f\right>\dt$ represents a
quadratic form of (the unique) bounded operator (denoted by)
$\int_E\AAA\dm$ or  $\int_E\AAA_t\dt$,
 (we refer to it  as to  ``intuitive''  integral
 of $\AAA$ over $E$),
 satisfying
\begin{equation}
  \left<\left(\int_E\AAA_t\dt\right) f,g\right>=
 \innE \left<\AAt f,g\right>\,d\mu (t) \textrm{\qquad for all  $f,g\in\HH$,}
\nonumber
\end{equation}
as well as
\begin{equation}
 \textrm{\rm tr}\!\left(\int_E\AAA_t\dt Y\right)=
 \innE \tr(\AAt Y)    \dt \textrm{\qquad for all  $Y\in\ccj$.}
\nonumber
\end{equation}
\end{lemma}

In other words, integrability of quadratic forms of an o.v. function
assures its Gel'fand's integrability and so the notions of
``intuitive'' and Gel'fand's integral for o.v. functions coincide.



Following Ex.2 in \cite{JOC},
for a \wmeasur 
 function
     $\AAA:\WWW\ttt \BH$  we have that $\AAA^*\AAA$ is Gel'fand
integrable iff $ \int_\Omega \lln\AAA_t f\rrn^2\dt<\iii$ for all
$f\in\HH$. Moreover,
for a \wmeasur 
 function
     $\AAA:\WWW\ttt \BH$ let us consider the operator of  ``vector valued functionalization'' (of vectors),
i.e., a linear transformation $\vec{\AAA}:D_{\vec{\AAA}}\ttt\LDH$,
with the domain $D_{\vec{\AAA}}=\llv f\in\HH \,| \,
  \int_\Omega \lln\AAA_t f\rrn^2\dt<\iii\rrv$, defined by
\begin{equation}
 ({\vec{\AAA}}f)(t)=\AAA_t f \qquad
\textrm{\rm for $[\mu]$ \,a.e.\, $t\in\WWW$ and all $f\in
 D_{\vec{\AAA}}$.}
\nonumber
\end{equation}
Now, another way for  understanding  Gel'fand's integrability of
$\AAA^*\AAA$ is provided by the following. \bbl
\label{zatvorenjeogranichen}
$\vec{\AAA}$ is a closed operator; it is
 bounded if and only if $\AAA^*\AAA$ is Gel'fand integrable,
and whenever this is the case, then $ \llaj\vec{\AAA}\rraj=\sqrt{
  \int_\WWW \AAA_t^*\AAA_t \dt}$ and
\begin{equation}
\llnj\vec{\AAA}\rrnj_{\BBBog(\HH,\LDH)}=
 \lln  \int_\WWW \AAA_t^*\AAA_t \dt\rrn^\frac12.
\nonumber
\end{equation}
If additionally $ \sqrt{ \int_\WWW \AAA_t^*\AAA_t \dt}\in
\ccc_{\Phi}(\HH)$, then
\begin{equation}
\llnj\vec{\AAA}\rrnj_{\ccc_{\Phi}(\HH,\LDH)}=
 \lln \sqrt{ \int_\WWW \AAA_t^*\AAA_t \dt}\rrn_{\ccc_{\Phi}(\HH)}.
\label{zatvogr2}
\end{equation}
\eel



Denoting $\lln\cdot\rrn_{\ccc_{\Phi}(\HH)}$ by $\llu\cdot\rru$, in
view of both Definition 1 and equality (3) in \cite{JOC} we now get
\begin{eqnarray}
\lluj{\AAA}\rruj_2&:=&\lln  \int_\WWW \AAA_t^*\AAA_t
\dt\rrn_{\ccc_{\Phi}(\HH)}^\frac12=
 \lln \sqrt{ \int_\WWW \AAA_t^*\AAA_t
 \dt}\rrn_{\ccc_{\Phi^{(2)}}(\HH)}\label{zatvogr3}\\ \nonumber
&=&\llnj \vec{\AAA}\rrnj_{\ccc_{\Phi^{(2)}}(\HH,\LDH)}.
\end{eqnarray}
Thus we have recognized the space $L_G^2(\WWW,\dm,\ccc_\Phi(\HH))$
of square integrable o.v. functions $\AAA$ such that
$\int_\WWW\AAA_t^*\AAA_t\dt\in \ccc_\Phi(\HH)$ as the isometrically
isomorph to the norm ideal of operators
$\ccc_{\Phi^{(2)}}(\HH,\LDH)\subseteq
 {\BBBog_{\Phi^{(2)}}(\HH,\LDH)}$, associated to a (2-reconvexized) s.g function
$\Phi^{(2)}$. Therefore the normability and the completeness of the
space $L_G^2(\WWW,\dm,\ccc_\Phi(\HH))$, as stated in Theorem 2.1 in
\cite{JOC}, follow immediately. This effectively gives us a
representation of $L_G^2(\WWW,\dm,\ccc_\Phi(\HH))$, as a Hilbert
module over a Banach $^*$-algebra ${\ccc_{\Phi}}(\HH)$ with its
${\ccc_{\Phi}}(\HH)$-valued inner product
$$\left<\!\left<\AAA,\BBB\right>\!\right>=
\int_\WWW \AAt^*\BBt\dt \qquad\textrm{\rm for all $\AAA,\BBB\in
L_G^2(\WWW,\dm,\ccc_\Phi(\HH))$,}$$ as a norm ideal
$\ccc_{\Phi^{(2)}}(\HH,\LDH)$ of operators from  $\HH$ into $\LDH$
equiped with its $\ccc_{\Phi}(\HH)$-valued inner product
$$\left<\!\!\left<\vec{\AAA},\vec{\BBB}\right>\!\!\right>=
\vec{\AAA^*}\vec{\BBB} =\int_\WWW \AAt^*\BBt\dt \qquad\textrm{\rm
for all $\vec{\AAA},\vec{\BBB} \in
 \ccc_{\Phi^{(2)}}(\HH,\LDH).$}$$

\subsection{Elementary operators and inner product type integral transformers in norm ideals}


For weakly$^*$-measurable o.v. functions $\AAA,\BBB:\WWW\to \BH$ and
for all $X\in\BH$ the function $t\to\AAA_t X\BBB_t$ is also
weakly$^*$-measurable. If these functions are Gel'fand integrable
for all $X\in\BH$, then the inner product type  linear
transformation $X\to\int_\WWW \AAA_t X\BBB_t\dt$ will be called an
inner product type (i.p.t.) transformer on $\BH$ and denoted by
$\int_\WWW \AAA_t \otimes\BBB_t\dt$ or ${\mathcal I}_{\AAA,\BBB}$. A
special case when  $\mu$ is the counting measure on $\mathbb N$ is
mostly known and widely investigated, and such transformers are
known as elementary mappings or elementary operators.

As shown in Lemma 3.1 (a) in \cite{JOC}, a sufficient condition is
provided when  $\AAA^*$ and $\BBB$ are both in
$L^2_G(\WWW,\dm,\BH).$ If each of families  $(\AAA_t)_{t\in\WWW}$
and  $(\BBB_t)_{t\in\WWW}$ consists of commuting normal operators,
then by Theorem 3.2 in \cite{JOC} the \ipt  transformer $\int_\WWW
\AAA_t \otimes\BBB_t\dt$ leaves every u.i. norm ideal
$\ccc_{\llu\cdot\rru}(\HH)$ invariant and the following
Cauchy-Schwarz inequality holds:
\begin{equation}
\llu \int_\WWW \AAA_t X\BBB_t\dt     \rru\le \llu \sqrt{\int_\WWW
\AAA_t^*\AAA_t \dt}    X
     \sqrt{\int_\WWW \BBB_t^*\BBB_t \dt}\rru
\end{equation}
for all $X\in \ccc_{\llu\cdot\rru}(\HH)$.

Normality and commutativity condition can be dropped for Schatten
$p$-norms as shown in Theorem 3.3 in \cite{JOC}. In Theorem 3.1 in
\cite{joc09i} a formula for the exact norm of the \ipt transformer
$\int_\WWW \AAA_t \otimes\BBB_t\dt$ acting on $\ccc_2(\HH)$ is
found. In Theorem 2.1 in \cite{joc09i}  the exact norm of the \ipt
transformer $\int_\WWW \AAA_t^* \otimes\AAA_t\dt$ is given for two
specific cases:
\begin{equation}
\lln \int_\WWW \AAA_t^*\otimes\AAA_t\dt
\rrn_{\BH\to\ccc_{\Phi}(\HH)}= \lln \int_\WWW \AAA_t^*\AAA_t\dt
\rrn_{\ccc_\Phi(\HH)}, \label{bhubiloshta}
\end{equation}
\begin{equation}
\lln \int_\WWW \AAA_t^*\otimes\AAA_t\dt
\rrn_{\ccc_{\Phi}(\HH)\to\ccc_1(\HH)}= \lln \int_\WWW
\AAA_t\AAA_t^*\dt     \rrn_{\ccc_{\Phi_*}(\HH)},
\nonumber
\end{equation}
where $\Phi_*$ stands for a s.g. function related to the dual space
$(\ccc_{\Phi}(\HH))^*$.

Also, as already noted in \cite{joc09i} at the end of page 2964, the
norm appearing in (\ref{zatvogr3}) equals to a square root of the
norm of the \ipt transformer $X\to\int_\WWW\AAA_t^* X\AAA_t\dt$ when
acting from $\BH$ to $\ccc_\Phi(\HH)$. As this quantity actually
presents a norm on the Banach space
$L_G^2(\WWW,\dm,\BH,\ccc_\Phi(\HH))$ as elaborated in Theorem 2.2 in
\cite{joc09i}, therefore we conclude that spaces
$L_G^2(\WWW,\dm,\BH,\ccc_\Phi(\HH))$ are
 both isometrically isomorph to the norm ideal
$\ccc_{\Phi^{(2)}}(\HH,\LDH).$ As the objects of consideration in
all those spaces are families of operators, from now  on we  will
refer to such objects as to {\bf field of operators} (for example
$(\AAA_t)_{t\in\WWW}$) in $L^2(\WWW,\mu,\ccc_\Phi(\HH))$. When  we
additionally require that the  adjoint field of operators
 $(\AAA_t^*)_{t\in\WWW}$  also belongs to
$L^2(\WWW,\mu,\ccc_\Phi(\HH))$, then we will say that
 $(\AAA_t)_{t\in\WWW}$  in doubly $\mu$ square integrable in
$\ccc_\Phi(\HH)$ on $\WWW.$

The norm appearing in (\ref{bhubiloshta})
and its associated space $L_G^2(\WWW,\dm,\BH,\ccc_\Phi(\HH))$
present only a  spacial case of norming a field
$\AAA=(\AAA_t)_{t\in\WWW}$. A much wider class of norms $ \lln
\cdot\rrn_{\Phi,\Psi}$ and their associated spaces
$L_G^2(\WWW,\dm,\ccc_\Phi(\HH),\ccc_\Psi(\HH))$ are
 given in
\cite{joc09i} by
\begin{equation}
  \lln \AAA\rrn_{\Phi,\Psi}=
\lln \int_\WWW \AAA_t^*\otimes \AAA_t\dt
 \rrn_{\BBBog(\ccc_\Phi(\HH),\ccc_\Psi(\HH))}^\frac12
\end{equation}
for an arbitrary pair of s.g. functions $\Phi$ and $\Psi$. For the
proof of completeness of the space
$L_G^2(\WWW,\dm,\ccc_\Phi(\HH),\ccc_\Psi(\HH))$ see Theorem 2.2 in
\cite{joc09i}.

The potential for finding  Gr\"uss type inequalities for  \ipt
transformers relies on the fact that $\int_\Omega
\AAA_t\otimes\BBB_td\mu(t)-\int_\Omega\AAA_t\dt \otimes
\int_\Omega\mathscr{B}_t \dt$ is also an   \ipt transformer. As the
representation for an
 \ipt transformer is not unique (as a rule), the successfulness of
the application of some  known inequalities to $\int_\Omega
\AAA_t\otimes\BBB_td\mu(t)-\int_\Omega\AAA_t\dt \otimes
\int_\Omega\BBB_t\dt$ mainly depends on the right  choice for its
representation.

Before exposing main results, we will draw our attention to the
following lemma, which we will use in the sequel.

\begin{lemma}
\label{optlema} If $\mu$ is a probability measure  on $\Omega$, then
for every field $(\mathscr{A}_t)_{t\in\Omega}$ in
$L^2(\Omega,\mu,\mathcalb{B}(\mathcal{H}))$, for all
$B\in\mathcalb{B}(\mathcal{H})$, for all unitarily invariant norms
$\lluo\cdot\rruo$ and for all $\theta>0$,
\begin{eqnarray}
\lefteqn{ \int_\Omega\left|\mathscr{A}_t-B\right|^2 \dt  =
 \int_\Omega\left|\mathscr{A}_t-\int_\Omega\AAA_t \dt\right|^2 \dt
+\lla \int_\Omega\AAA_t \dt-B\rra^2}\label{nulto}\\
 &\ge& \int_\Omega\left|\mathscr{A}_t-\int_\Omega\AAA_t
\dt\right|^2 \dt =\int_\Omega|\mathscr{A}_t|^2
\dt-\left|\int_\Omega\AAA_t \dt\right|^2; \label{prvo}
\end{eqnarray}
\begin{multline}
\min_{B\in\mathcalb{B}(\mathcal{H})}\llu \,\left|\int_\Omega\left|\mathscr{A}_t-B\right|^2\dt\rrr|^\theta\rru\\
=  \llu\, \left|\int_\Omega\left|\mathscr{A}_t-\int_\Omega\AAA_t
\dt\right|^2 \dt\right|^\theta \rru=\llu \,
\left|\int_\Omega|\mathscr{A}_t|^2 \dt- \left|\int_\Omega\AAA_t
\dt\right|^2\rrr|^\theta\rru. \label{drugo}
\end{multline}
\end{lemma}
Thus, the considered minimum is always obtained for
$B=\int_\Omega\AAA_t \dt$.

\begin{proof}
The expression in (\ref{nulto}) equals
\begin{eqnarray}
\lefteqn{ \int_\Omega\left|\mathscr{A}_t-B\right|^2 \dt  =
  \int_\Omega\left|\mathscr{A}_t-\int_\Omega\AAA_t \dt+\int_\Omega\AAA_t \dt-B\right|^2\dt=}\nonumber\\
&=& \int_\Omega\left|\mathscr{A}_t-\int_\Omega\AAA_t\dt\right|^2 \dt
 + \int_\WWW\lla \int_\WWW\mathscr{A}_t\dt-B\rra^2 \dt\nonumber\\
&+&2\Re \int_\Omega\llm \mathscr{A}_t-\int_\Omega\AAA_t\dt\rrm^*
\llm\int_\Omega\AAA_t \dt-B\rrm\dt\nonumber\\
&=& \int_\Omega\left|\mathscr{A}_t-\int_\Omega\AAA_t\dt\right|^2 \dt
+\lla\int_\Omega\mathscr{A}_t\dt-B\rra^2,
\nonumber
\end{eqnarray}
as $\displaystyle\int_\Omega\llm
\mathscr{A}_t-\int_\Omega\AAA_t\dt\rrm^{*}\llm\int_\Omega\AAA_t
\dt-B\rrm\dt =$ $$=\llm\int_\Omega
\mathscr{A}_t^*\dt-\int_\Omega\AAA_t^*\dt\rrm\llm\int_\Omega\AAA_t
\dt-B\rrm=0.$$

Inequality in  (\ref{prvo}) follows from (\ref{nulto}), while
identity in \eqref{prvo} is just a
 a special case of  Lemma 2.1 in \cite{JOC} applied for $k=1$ and $\delta_1=\Omega$.

As $0\le A\le B$ for $A,B\in\cci$ implies $ s_n^\theta(A)\le
s_n^\theta(B)$ for all $n\inN$, as well as $\llu A^\theta\rru\le
\llu B^\theta\rru,$ then (\ref{drugo}) follows.
\end{proof}

\section{Main results}

Let us recall that for a pair of random real variables $(Y,Z)$ its
coefficient  of correlation
$$\rho_{Y,Z}=\frac{\lla E(YZ)-E(Y)E(Z)\rra}{\sigma(Y)\sigma(Z)}=
             \frac{\lla E(YZ)-E(Y)E(Z)\rra}{
\sqrt{E(Y^2)-E^2(Y)} \sqrt{E(Z^2)-E^2(Z)}}$$ always satisfies
$|\rho_{Y,Z}|\le 1.$

For square integrable functions $f$ and $g$ on $[0,1]$ and
$D(f,g)=\int_0^1f(t)g(t)\,d t-
        \int_0^1f(t)\,d t\int_0^1g(t)\,d t$
Landau proved (see  \cite{lan1,lan2})
$$ \lla D(f,g)\rra\le \sqrt{D(f,f)D(g,g)},$$
and the  following theorem is a generalization of these facts to
\ipt{} transformers.

\begin{theorem}[Landau type inequality for \ipt  transformers in u.i. norm ideals]
\label{normalnisluchaj} If $\mu$ is a probability measure on
$\Omega$, let both fields $(\AAA_t)_{t\in\Omega}$ and
 $(\BBB_t)_{t\in\Omega}$ be in $L^2(\Omega,\mu,\BH)$
 consisting of commuting normal operators
and let
$$\sqrt{\,\int_\Omega|\AAt|^2 \dt-\left|\int_\Omega\AAt \dt\right|^2}X
\sqrt{\,\int_\Omega|\BBt|^2 \dt-\left|\int_\Omega\BBt
\dt\right|^2}\in \ccu$$ for some $X\in\BH$. Then $$\int_\Omega
\AAA_tX\BBB_td\mu(t)-\int_\Omega\AAt \dt X\!\!\int_\Omega\BBt \dt
\in\ccu$$ and
\begin{multline}
\llu\int_\Omega \AAA_t X\BBB_td\mu(t)- \int_\Omega\AAt \dt X\!
\!\int_\Omega\BBt \dt
\rru\\
\leq\llu
  \sqrt{\,\int_\Omega|\AAt|^2 \dt-\left|\int_\Omega\AAt \dt\right|^2}X
\sqrt{\,\int_\Omega|\BBt|^2 \dt-\left|\int_\Omega\BBt \dt\right|^2}
\rru.\label{21}
\end{multline}
\end{theorem}

\begin{proof}
First we note that we have the following Korkine type identity for
\ipt transformers:
\begin{eqnarray}
\nonumber && \int_\Omega \AAA_tX\BBB_td\mu(t)\!-\!\int_\Omega\AAt
\dt X\! \!     \int_\Omega\BBt \dt
\!\!\\
\nonumber&&\lefteqn{=\!\!\int_\Omega d\mu(s)\int_\Omega \AAA_tX\BBB_t \dt\!-\!\int_\Omega\!\int_\Omega \AAA_tX\BBB_s\,d\mu(s)d\mu(t)}\\
\!\!&&=\!\!\dfrac12\int_{\Omega^2}(\AAA_s-\AAA_t)X(\BBB_s-\BBB_t)d(\mu\times\mu)(s,t).\label{22}
\end{eqnarray}

In this representation we have $(\AAA_s-\AAA_t)_{(s,t)\in\Omega^2}$
and $(\BBB_s-\BBB_t)_{(s,t)\in\Omega^2}$ to be in
$L^2(\Omega^2,\mu\times\mu,\BH)$ because by  an application of the
identity (\ref{22}),

\begin{eqnarray}
\nonumber
\dfrac12\int_{\Omega^2}\left|\AAA_s-\AAA_t\right|^2d(\mu\times\mu)(s,t)&=&\int_\Omega|\AAt|^2
\dt-
\left|\int_\Omega\AAt \dt\right|^2\\
&=&\int_\Omega\left|\AAA_t-\int_\Omega\AAt \dt\right|^2 \dt
\in\BH.\label{23}
\end{eqnarray}

Both families $(\AAA_s-\AAA_t)_{(s,t)\in\Omega^2}$ and
$(\BBB_s-\BBB_t)_{(s,t)\in\Omega^2}$ consist of commuting normal
operators and by Theorem 3.2 in \cite{JOC}
 $$\dfrac12\int_{\Omega^2}(\AAA_s-\AAA_t)X(\BBB_s-\BBB_t)d(\mu\times\mu)(s,t)\in \ccc_{\lluo\cdot|
\rruo}(\HH)\qquad\mbox{and}$$

\begin{eqnarray*}
\lefteqn{ \llu             \int_\Omega \AAA_tX\BBB_t \dt-
\int_\Omega\AAt \dt X\!\!\int_\Omega\BBt \dt
\rru}\\
&=&\llu\dfrac12\int_{\Omega^2}(\AAA_s-\AAA_t)X(\BBB_s-\BBB_t)d(\mu\times\mu)(s,t)\rru\\
&\le&\llu\sqrt{\,\dfrac12\int_{\Omega^2}|\AAA_s-\AAA_t|^2d(\mu\times\mu)(s,t)}
X\sqrt{\,\dfrac12\int_{\Omega^2}|\BBB_s-\BBB_t|^2d(\mu\times\mu)(s,t)}\rru\\
&=&\llu\sqrt{\,\int_\Omega|\AAt|^2 \dt-\left|\int_\Omega\AAt
\dt\right|^2} X\sqrt{\,\int_\Omega|\BBt|^2 \dt-\left|\int_\Omega\BBt
\dt\right|^2}\rru,
\end{eqnarray*}
due to identities (\ref{22}) and (\ref{23}). And so the conclusion
(\ref{21}) follows.
\end{proof}

\begin{lemma}
Let $\mu$ (resp. $\nu$) be  a probability measure on $\Omega$ (resp.
$\mho$), let both families
$\{\AAA_s,\CCC_t\}_{(s,t)\in\Omega\times\mho}$ and
$\{\BBB_s,\DDD_t\}_{(s,t)\in\Omega\times\mho}$
 consist of commuting normal operators
and let
\begin{eqnarray*}&&\sqrt{\,\int_\Omega|\AAA_s|^2d\mu(s)
\int_\mho|\CCC_t|^2d\nu(t)-\left|\int_\Omega\AAA_sd\mu(s)\int_\mho\CCC_td\nu(t)\right|^2}\cdot X\cdot\\
&&\sqrt{\,\int_\Omega|\BBB_s|^2d\mu(s)
\int_\mho|\DDD_t|^2d\nu(t)-\left|\int_\Omega\BBB_sd\mu(s)\int_\mho\DDD_td\nu(t)\right|^2}
\end{eqnarray*}
be in $\ccu$ for some $X\in\BH$.
Then \begin{eqnarray*}&&\int_\Omega \int_\mho\AAA_s
\CCC_tX\BBB_s\DDD_t\,d\mu(s)\,d\nu(t) -\\&-&\int_\Omega\AAA_s
\,d\mu(s)\int_\mho\CCC_t\,d\nu(t) X\!\!\int_\Omega\BBB_s \,d\mu(s)
\int_\mho\DDD_t\,d\nu(t) \in\ccu\end{eqnarray*}
and
$$
\llu \!\int_\Omega \!\int_\mho\AAA_s
\CCC_tX\BBB_s\DDD_t\,d\mu(s)\,d\nu(t) -\!\int_\Omega\AAA_s
\,d\mu(s)\!\int_\mho\CCC_t\,d\nu(t) X\!\!\int_\Omega\BBB_s
\,d\mu(s)\kern-2pt \int_\mho\DDD_t\,d\nu(t) \rru $$
\begin{eqnarray}
&  \leq&\llu \sqrt{\!\int_\Omega|\AAA_s|^2d\mu(s)
        \!\int_\mho|\CCC_t|^2d\nu(t)\!-\!\left|\!\int_\Omega \AAA_s d\mu(s)\!\int_\mho\CCC_t d\nu(t)\right|^2}\right.\right.\right.X\nonumber\\
&\cdot& \left.\left.\left.\sqrt{\!\int_\Omega|\BBB_s|^2d\mu(s)
        \!\int_\mho|\DDD_t|^2d\nu(t)\!-\!\left|\!\int_\Omega \BBB_s d\mu(s)\!\int_\mho\DDD_t d\nu(t)\right|^2}\rru.
\nonumber
\end{eqnarray}
\end{lemma}

\begin{proof}
Apply Theorem \ref{normalnisluchaj} to the probability measure
$\mu\times\nu$ on $\WWW\times\mho$ and families
 $(\AAA_s\CCC_t)_{(s,t)\in\Omega\times\mho}$ and
$(\BBB_s\DDD_t)_{(s,t)\in\Omega\times\mho}$ of normal commuting
operators in $L_G^2(\Omega\times\mho,d\mu\times\nu,\BH),$ taking in
account that
\begin{eqnarray*} \int_{\WWW\times\mho}\AAA_s \CCC_t\,d(\mu\times\nu)(s,t)
&=&\int_\Omega\AAA_s \,d\mu(s)\!\int_\mho\CCC_t\,d\nu(t),\\
   \int_{\WWW\times\mho}\BBB_s \DDD_t\,d(\mu\times\nu)(s,t)
&=&\int_\Omega\BBB_s
\,d\mu(s)\!\int_\mho\DDD_t\,d\nu(t),\end{eqnarray*} and similarly
\begin{eqnarray*} \int_{\WWW\times\mho}\lla\AAA_s\CCC_t\rra^2d(\mu\times\nu)(s,t)
&=&\int_\Omega\lla\AAA_s\rra^2d\mu(s)\int_\mho\lla\CCC_t\rra^2d\nu(t),\\
   \int_{\WWW\times\mho}\lla\BBB_s \DDD_t\rra^2d(\mu\times\nu)(s,t)
&=&\int_\Omega\lla\BBB_s
\rra^2d\mu(s)\int_\mho\lla\DDD_t\rra^2d\nu(t).\end{eqnarray*}
\end{proof}

By the use of the mathematical induction, the previous lemma enables
us to get straightforwardly the following
\begin{corollary}
If $\mu$, $(\AAA_t)_{t\in\Omega}$ and
 $(\BBB_t)_{t\in\Omega}$ are as in Theorem \ref{normalnisluchaj},
then for all $n\in\mathbb N$ and for any
 $(*_1,\cdots,*_{2n})\in\{*,1\}^{2n}$,
\begin{eqnarray}
\lefteqn{
    \llu\int_{\Omega^n}\prod_{k=1}^n \AAA_{t_k}^{*_k}X
   \prod_{k=1}^n \BBB_{t_{n+k}}^{*_{n+k}}    \prod_{k=1}^{n} d\mu(t_k)\rrr.\rrr.\rrr.}\nonumber\\
& - &\lll.\lll.\lll.  \llm\int_{\Omega}\AAA_{t}^*\dt\rrm^i
 \llm\int_{\Omega}\AAA_{t}\dt\rrm^{n-i}X
 \llm\int_{\Omega}\BBB_{t}^*\dt\rrm^j
 \llm\int_{\Omega}\BBB_{t}\dt\rrm^{n-j}\rru\nonumber
\end{eqnarray}
$$\le
\llu \llm\int_\Omega|\AAt|^2 \dt-\left|\int_\Omega\AAt
\dt\right|^2\rrm^{\frac{n}2}X \llm\int_\Omega|\BBB_t|^2
\dt-\left|\int_\Omega\BBB_t \dt\right|^2\rrm^{\frac{n}2} \rru,$$

where $i$ (resp. $j$) stands for the cardinality of $\{k\in\mathbb N
\,\,|\, 1\le k\le n \,\,\& \,*_k=*\}$ and (resp. $\{k\in\mathbb N
\,\,|\, 1\le k\le n \,\,\& \,*_{n+k}=*\}$).
\end{corollary}

For the Schatten $p$-norms $\|\cdot\|_p$ normality and commutativity
conditions can be dropped, at the  inevitable expense of the
simplicity for its formulation. So we have the following

\begin{theorem}[Landau type inequality for \ipt  transformers in Schatten ideals]
Let $\mu$ be a probability measure on $\Omega$, let
$(\AAA_t)_{t\in\Omega}$ and $(\BBB_t)_{t\in \Omega}$ be
$\mu$-weak${}^*$ measurable families of bounded Hilbert space
operators such that
$$\int_\Omega\left(\|\AAA_tf\|^2+\|\AAA_t^*f\|^2+\|\BBB_tf\|^2+\|\BBB_t^*f\|^2\right)d\mu(t)<\infty\qquad
\textrm{ \rm for all $f\in\HH$}$$ and let $p,q,r\ge1$ such that
$\dfrac1p=\dfrac1{2q}+\dfrac1{2r}\,$. Then for all
$X\in\ccc_p(\HH)$,

\begin{eqnarray}
\label{grussp}&&\lefteqn{
           \lln\int_\Omega \AAA_tX\BBB_t \dt-      \int_\Omega\AAt \dt   X \int_\Omega\BBt \dt\rrn_p}\\
         &\leqslant&\kern-4pt\lln
\left(\int_\Omega\left|\left(\int_\Omega\left|\AAA_t^*-\int_\Omega\AAt^*
\dt \right|^2 \dt   \right)^{\frac{q-1}2}
\left(\AAA_t-\int_\Omega\AAt \dt\right)\right|^2
\dt\right)^{\frac1{2q}}\rrr.\nonumber\end{eqnarray}
\begin{equation*} X\lll.\left(\int_\Omega\left|\left(\int_\Omega\left|\BBB_t-\int_\Omega\BBt\dt
\right|^2 \dt   \right)^{\frac{r-1}2}
\left(\BBB_t^*-\int_\Omega\BBt^* \dt\right)\right|^2
\dt\right)^{\frac1{2r}}\rrn_p.
\end{equation*}

\end{theorem}

\begin{proof}

According to identity (\ref{23}), application of Theorem 3.3 in
\cite{JOC} to   families
$(\mathscr{A}_s-\mathscr{A}_t)_{(s,t)\in\Omega^2}$ and
 $(\mathscr{B}_s-\mathscr{B}_t)_{(s,t)\in\Omega^2}$ gives

\begin{eqnarray}
&&
 \lln\int_\Omega \AAA_tX\BBB_td\mu(t)-\int_\Omega\mathscr{A}_t\dt X\int_\Omega\mathscr{B}_t\dt\rrn_p \nonumber\\
&=&\lln\dfrac12\int_{\Omega^2}(\AAA_s-\AAA_t)X(\BBB_s-\BBB_t)d(\mu\times\mu)(s,t)\rrn_p\le \nonumber
\end{eqnarray}

$$
\lln\left(\dfrac12\int_{\Omega^2}(\mathscr{A}_s^*-\mathscr{A}_t^*)
\left(\dfrac12\int_{\Omega^2}|\mathscr{A}_s^*-\mathscr{A}_t^*|^2(\mu\times\mu)(s,t)\right)^{q-1}\kern-13.1pt
(\mathscr{A}_s-\mathscr{A}_t)d(\mu\times\mu)(s,t)\right)^{\frac1{2q}}\rrr.\kern-10pt
X$$ \begin{equation}\label{pnorm}\end{equation}
$$\lll.\left(\dfrac12\int_{\Omega^2}(\mathscr{B}_s-\mathscr{B}_t)
\Bigl(\dfrac12\int_{\Omega^2}|\mathscr{B}_s-\mathscr{B}_t|^2(\mu\times\mu)(s,t)\Bigr)^{r-1}\kern-8pt
(\mathscr{B}_s^*-\mathscr{B}_t^*)d(\mu\times\mu)(s,t)\right)^{\frac1{2r}}\kern-2pt\rrn_p.$$

By application of identity (\ref{23}) once again, the last
expression in (\ref{pnorm}) becomes

$$\bigg\|\biggl(\dfrac12\int_{\Omega^2}(\mathscr{A}_s-\mathscr{A}_t)^*
\left(\int_\Omega\left|\mathscr{A}_t^*-\int_\Omega \mathscr{A}^*
_t\dt\right|^2\dt\right)^{q-1}
(\mathscr{A}_s-\mathscr{A}_t)d(\mu\times\mu)(s,t)\biggr)^{\frac1{2q}}$$

$$X\biggl(\dfrac12\int_{\Omega^2}(\mathscr{B}_s-\mathscr{B}_t)
\Bigl(\int_\Omega\left|\mathscr{B}_s-\int_\Omega \mathscr{B}
_t\dt\right|^2d\mu(s)\Bigr)^{r-1}\kern-5pt(\mathscr{B}_s-\mathscr{B}_t)^*d(\mu\times\mu)(s,t)\biggr)^{\frac1{2r}}\bigg\|_p.\label{odvizraz}
$$

Denoting
$\Bigl(\int_\Omega\left|\AAA_s^*-\int_\Omega\mathscr{A}^*d\mu\right|^2d\mu(s)\Bigr)^{\frac{p-1}2}$
\kern-6.5pt(resp.
$\Bigl(\int_\Omega\left|\BBB_s-\int_\Omega\mathscr{B}d\mu\right|^2d\mu(s)\Bigr)^{\frac{r-1}2}$)
by $Y$ (resp. $Z$),
then 
the expression in (\ref{pnorm}) becomes

\begin{eqnarray}\label{saYZ}
\biggl\|\left(\dfrac12\int_{\Omega^2}\left|Y\AAA_s-Y\AAA_t\right|^2d(\mu\times\mu)(s,t)\right)^{\frac1{2q}}X\\
\nonumber
\left(\dfrac12\int_{\Omega^2}\left|Z\BBB_s^*-Z\BBB_t^*\right|^2d(\mu\times\mu)(s,t)\right)^{\frac1{2r}}\biggr\|_p.
\end{eqnarray}

By a new application of identity (\ref{23}) to families
$(Y\AAA_t)_{t\in\Omega}$ and $(Z\BBB_t^*)_{t\in\Omega}$ (\ref{saYZ})
becomes

$$\left\|\left(\int_\Omega\left|Y\mathscr{A}_t-\int_\Omega Y\mathscr{A}_t \dt\right|^2d\mu(t)\right)^{\frac1{2q}}\kern-4pt X
\left(\int_\Omega\left|Z\mathscr{B}_t^*-\int_\Omega Z\mathscr{B}_t^*
\dt\right|^2d\mu(t)\right)^{\frac1{2r}}\right\|_p,$$

which obviously equals to the righthand side expression in
(\ref{grussp}).
\end{proof}
A special case of an abstract H\"older inequality presented in
Theorem 3.1.(e) in \cite{JOC}
 gives us

\begin{theorem}[Cauchy-Schwarz  inequality for o.v. functions in u.i. norm ideals]
\label{koshishvarcovacha} Let $\mu$ be a  measure on $\Omega$, let
$(\AAA_t)_{t\in\Omega}$
 and $(\BBB_t)_{t\in \Omega}$ be
$\mu$-weak${}^*$ measurable families of bounded Hilbert space
operators
 such that
$\lla\int_\Omega|\AAA_t|^2\dt\rra^\theta$ and
$\lla\int_\Omega|\BBB_t|^2\dt\rra^\theta$ are in in $\ccu$ for some
$\theta>0$ and for some u.i. norm $\lluo\cdot\rruo.$ Then,
\begin{equation}
           \llu\lla\int_\Omega \AAA_t^*\BBB_t\dt\rra^\theta \rru\le
           \llu\lla\int_\Omega \AAA_t^*\AAA_t\dt\rra^\theta \rru^\frac12
           \llu\lla\int_\Omega \BBB_t^*\BBB_t\dt\rra^\theta \rru^\frac12.
\nonumber
\end{equation}
\end{theorem}
\begin{proof}
Take   $\Phi$ to be a
     s.g. function that generates u.i. norm $\lluo\cdot\rruo$,
 $\Phi_1=\Phi$,
$\Phi_2=\Phi_3=\Phi^{(2)}$ (2-reconvexization  of $\Phi$),
$\alpha=2\theta$ and $X=I$, and then apply 3.1.(e) from \cite{JOC}.
\end{proof}

Now we can easily derive the following generalization of Landau
inequality  for Gel'fand integrals of o.v. functions:
\begin{theorem}[Landau type inequality for o.v. functions in u.i. norm ideals]
\label{korelacionateorema} If  $\mu$ is  a probability  measure on
$\Omega$, $\theta>0$ and $(\AAA_t)_{t\in\Omega}$
 and $(\BBB_t)_{t\in \Omega}$ are as
   in Theorem \ref{koshishvarcovacha},
then,
\begin{multline}
           \llu\lla\int_\Omega \AAA_t^*\BBB_t\dt
                  -\int_\Omega \AAA_t^*\dt\int_\WWW\BBB_t\dt\rra^\theta \rru^2\le\\
           \llu\lla\int_\Omega \lla\AAA_t\rra^2\dt-\lla\int_\WWW\AAt\dt\rra^2\rra^\theta\rru
           \llu\lla\int_\Omega \lla\BBB_t\rra^2\dt-\lla\int_\WWW\BBt\dt\rra^2\rra^\theta\rru.
\label{korelaciona}
\end{multline}
\end{theorem}
\begin{proof}
Apply Theorem \ref{koshishvarcovacha}  to o.v. families
$(\AAA_s-\AAA_t)_{(s,t)\in\WWW^2}$ and
$(\BBB_s-\BBB_t)_{(s,t)\in\WWW^2}$ and use identity (\ref{22}) once
again.
\end{proof}

For for the more general inequality in  an arbitrary Hilbert
$C^*$-module see Theorem 3.4 in \cite{I-V}. A case $\theta=1$ and
$\lluo\cdot\lluo=\lln\cdot\rrn$ in Theorem \ref{korelacionateorema}
offers the proof for Hilbert $C^*$-module $L_G^2(\WWW,\dm,\BH)$ in
case of the lifted projection $h(t)=I$ for all $t\in\WWW$.

Case $\theta=1$ and $\lluo\cdot\lluo=\lln\cdot\rrn_1$ of Theorem
\ref{korelacionateorema}
offers the proof for the stronger version of Theorem 3.3 for Hilbert
$H^*$-module $L_G^2(\WWW,\dm,\ccj)$ for the
 same  lifted projection $h(t)=I$ for all $t\in\WWW$.

\begin{corollary}
Under conditions of Theorem \ref{korelacionateorema} we have

$$ \lla\tr\lll(    \int_\Omega \AAA_t^*\BBB_t\dt
                  -\int_\Omega \AAA_t^*\dt\int_\WWW\BBB_t\dt\rrr)\rra^2$$
\begin{eqnarray}&\le&         \lln    \int_\Omega \AAA_t^*\BBB_t\dt
                  -\int_\Omega \AAA_t^*\dt\int_\WWW\BBB_t\dt     \rrn_1^2 \label{druga}
\end{eqnarray}
\begin{eqnarray}
&\le&\left({\int_\Omega \lln\AAA_t\rrn_2^2\dt}-\lln\int_\WWW\AAt\dt\rrn_2^2\right)\cdot\nonumber\\
&\cdot&
   \left({\int_\Omega \lln\BBB_t\rrn_2^2\dt}-\lln\int_\WWW\BBt\dt\rrn_2^2\right)\nonumber\\
&=&\left(\lln\sqrt{\int_\Omega \lla\AAA_t\rra^2\dt}\rrn_2^2-\lln\int_\WWW\AAt\dt\rrn_2^2\right)\cdot\nonumber\\
&\cdot&\left(\lln\sqrt{\int_\Omega
\lla\BBB_t\rra^2\dt}\rrn_2^2-\lln\int_\WWW\BBt\dt\rrn_2^2\right).\nonumber
\label{korelaciona1}
\end{eqnarray}
\end{corollary}
\begin{proof}
An application  of (\ref{korelaciona}) for $\theta=1$ and
$\lluo\cdot\lluo=\lln\cdot\rrn_1$ justifies (\ref{prva}), while
(\ref{druga}) and all the remaining identities in
(\ref{korelaciona2}) are obtainable by a straightforward
calculations, based on elementary properties of the trace $\tr$ and
Gel'fand integrals:
\begin{eqnarray}
&&\lln\int_\Omega \AAA_t^*\BBB_t\dt
                  -\int_\Omega \AAA_t^*\dt\int_\WWW\BBB_t\dt \rrn_1^2 \nonumber \\
&\le&\lln\int_\Omega
\lla\AAA_t\rra^2\dt\kern-3pt-\kern-3pt\lla\int_\WWW\AAt\dt\rra^2\rrn_1
     \kern-4pt\lln\int_\Omega \lla\BBB_t\rra^2\dt\kern-3pt-\kern-3pt\lla\int_\WWW\BBt\dt\rra^2\rrn_1\label{prva}\\
 &=&\left(\tr\lll(\int_\Omega \lla\AAA_t\rra^2\dt\rrr)-\tr\lll(\lla\int_\WWW\AAt\dt\rra^2\rrr)\right)\cdot
\nonumber\\
&\cdot&
 \left(\tr\lll(\int_\Omega \lla\BBB_t\rra^2\dt\rrr)-\tr\lll(\lla\int_\WWW\BBt\dt\rra^2\rrr)\right)\nonumber\\
&=&\left({\int_\Omega \lln\AAA_t\rrn_2^2\dt}-\lln\int_\WWW\AAt\dt\rrn_2^2\right)\cdot\nonumber\\
&\cdot&
   \left({\int_\Omega \lln\BBB_t\rrn_2^2\dt}-\lln\int_\WWW\BBt\dt\rrn_2^2\right)\nonumber\\
&=&\left(\lln\sqrt{\int_\Omega \lla\AAA_t\rra^2\dt}\rrn_2^2-\lln\int_\WWW\AAt\dt\rrn_2^2\right)\cdot\nonumber\\
&\cdot&   \left(\lln\sqrt{\int_\Omega
\lla\BBB_t\rra^2\dt}\rrn_2^2-\lln\int_\WWW\BBt\dt\rrn_2^2\right).
\label{korelaciona2}
\end{eqnarray}
\end{proof}

For bounded field of operators $\AAA=(\mathscr{A}_t)_{t\in\Omega}$
one can easily check that  the radius of the smallest disk that
essentially contains its range is
$$r_\iii(\AAA)=\inf_{A\in\BH}\supess_{t\in\WWW}\lln \AAt-A\rrn=
\inf_{A\in\BH}\lln\AAA-A\rrn_\infty=\min_{A\in\BH}\lln\AAA-A\rrn_\infty$$
(from the triangle inequality we have
$\bigl|\|\mathscr{A}_t-A'\|-\|\mathscr{A}_t-A\|\bigr|\leq\|A'-A\|$,
so the mapping $A\to\supess_{t\in\WWW}\lln\AAt-A\rrn$ is nonnegative
and continuous on $\BH$; since $(\mathscr{A}_t)_{t\in\Omega}$ is
bounded field of operators, we also have $\lln\AAt-A\rrn\to\infty$
when $\|A\|\to\infty$, so this mapping attains minimum), and it
actually  attains at some $A_0\in\BH$, which represents a center of
the disk considered. Any such field of operators is of finite
diameter
$$\diam\nolimits_\iii(\AAA)=\supess_{s,t\in\WWW}\lln \AAA_s-\AAt\rrn,$$ with the simple
inequalities $r_\iii(\AAA)\le \diam_\iii(\AAA)\le 2r_\iii(\AAA)$
relating those quantities. For such
 fields of operators we can now state the following stronger
version of Gr\"uss inequality.

\begin{theorem}[Gr\"uss type inequality for \ipt{} transformers in u.i. norm ideals]
\label{th0} Let  $\mu$ be a $\sigma$-finite measure on $\Omega$ and
 let $\AAA=(\mathscr{A}_t)_{t\in\Omega}$ and $\BBB=(\mathscr{B}_t)_{t\in\Omega}$
be $[\mu]$ a.e. bounded  fields of operators. Then for all
$X\in\ccu$,

\begin{multline}
\sup_{\mu(\delta)>0}\llu\frac1{\mu(\delta)}\int_\delta\mathscr{A}_tX\mathscr{B}_t\dt-
\frac1{\mu(\delta)}\int_\delta\mathscr{A}_t\dt \,X
\frac1{\mu(\delta)}\int_\delta\mathscr{B}_t \dt\rru\\
\le\min \llv {r_\iii(\AAA)r_\iii(\BBB)},
             \frac{\diam_\iii(\AAA)\diam_\iii(\BBB)}2\rrv\cdot\lluo X\rruo
\label{oshtrina0}
\end{multline}
(i.e. $\sup$ is taken over all measurable sets
$\delta\subseteq\Omega$ such that $0<\mu(\delta)<\infty$).
\end{theorem}

\begin{proof}
Let
$r_\iii(\AAA)=\displaystyle\lln\AAA-A_0\rrn_\infty=\min_{A\in\BH}\lln
\AAA-A\rrn_\infty$, let
$r_\iii(\BBB)=\displaystyle\lln\BBB-B_0\rrn_\infty$
$\displaystyle=\min_{B\in\BH}\lln \BBB-B\rrn_\infty$ and let us note
that
\begin{eqnarray*}
\frac1{\mu(\delta)}\int_\delta\left|\mathscr{A}_t-A_0\right|^2\dt&\le&
\frac1{\mu(\delta)}
\int_\delta\supess_{t\in\WWW}\lln\mathscr{A}_t-A_0\rrn^2
\cdot I\dt\\
&=&\lln \mathscr{A}-A_0\rrn^2_\iii\cdot I=r_\iii^2(\AAA)\cdot
I.\end{eqnarray*} Therefore
$\lln\frac1{\mu(\delta)}\int_\delta\left|\mathscr{A}_t-A_0\right|^2\dt\rrn^\frac12
\le r_\iii(\AAA) $ and
$\lln\frac1{\mu(\delta)}\int_\delta\left|\mathscr{B}_t-B_0\right|^2\dt\rrn^\frac12
\le r_\iii(\BBB) $
 goes similarly.
By identity (\ref{22}) applied to probability measure
$\frac{1}{\mu(\delta)}\mu$ on $\delta$ and Lemma \ref{optlema}  we
then  have
$$
\frac1{2\mu(\delta)^2}\int_{\delta^2}\left|\mathscr{A}_s-\mathscr{A}_t\right|^2d(\mu\times\mu)(s,t)
=\frac1{\mu(\delta)}\int_{\delta}\left|\mathscr{A}_t-\frac1{\mu(\delta)}\int_\delta
\mathscr{A}_t\dt\rra^2\,d\mu(t)=$$
\begin{eqnarray*}&=&\frac1{\mu(\delta)}\int_\delta\left|\mathscr{A}_t-A_0\right|^2-\lla\frac1{\mu(\delta)} \int_\delta\AAA_t\dt-A_0\rra^2
d\mu(t)\\
&\le& \lln \mathscr{A}-A_0\rrn^2_\iii\cdot I-\lla\frac1{\mu(\delta)}
\int_\delta\AAA_t\dt-A_0\rra^2
\end{eqnarray*}
and therefore
\begin{eqnarray}
&&\lln\frac1{2\mu(\delta)^2}\int_{\delta^2}\left|\mathscr{A}_s-\mathscr{A}_t\right|^2d(\mu\times\mu)(s,t)\rrn\label{poluprechnik1}\\
&\le& \lln \lln\mathscr{A}-A_0\rrn^2_\iii\cdot I-\lla
\frac1{2\mu(\delta)}\int_\delta\AAA_t\dt-A_0\rra^2\rrn.\nonumber
\end{eqnarray}
Similarly,
\begin{eqnarray}
&&\lln\frac1{\mu(\delta)^2}\int_{\delta^2}\left|\mathscr{B}_s-\mathscr{B}_t\right|^2d(\mu\times\mu)(s,t)\rrn\label{poluprechnik2}\\
&\le&\lln \lln\mathscr{B}-B_0\rrn^2_\iii\cdot I-\lla
\frac1{\mu(\delta)}\int_\delta\BBB_t\dt-B_0\rra^2\rrn.\nonumber
\end{eqnarray}

Those inequalities show that 
subfields
$(\mathscr{A}_t-\mathscr{A}_s)_{(s,t)\in\delta\times\delta}$ and
$(\mathscr{B}_t-\mathscr{B}_s)_{(s,t)\in\delta\times\delta}$ are in
$L^2(\delta\times\delta,\frac1{\mu(\delta)}\mu\times\frac1{\mu(\delta)}\mu,\BH)$,
and therefore according to  identity (\ref{22}) and Lemma 3.1(c)
from \cite{JOC},

\begin{eqnarray}
&&\lefteqn{\llu \frac1{\mu(\delta)}\int_\delta \AAA_tX\BBB_td\mu(t)-
\frac1{\mu(\delta)}\int_\delta\mathscr{A}_t\dt X  \frac1{\mu(\delta)}\int_\delta\mathscr{B}_t \dt\rru}\nonumber\\
&=&\llu\frac1{2\mu(\delta)^2}\int_{\delta^2}(\AAA_s-\AAA_t)X(\BBB_s-\BBB_t)d(\mu\times\mu)(s,t)\rru\nonumber\\
&\le& \lln\frac1{2\mu(\delta)^2}\int_{\delta^2}\left|\mathscr{A}_s-\mathscr{A}_t\right|^2d(\mu\times\mu)(s,t)\rrn^\frac12\cdot\nonumber\\
&&
\cdot\lln\frac1{2\mu(\delta)^2}\int_{\delta^2}\left|\mathscr{B}_s-\mathscr{B}_t\right|^2d(\mu\times\mu)(s,t)\rrn^\frac12
\!\!\lluo X\rruo\nonumber\\
&\le&
 \lln\lln\mathscr{A}-A_0\rrn^2_\iii\cdot I
 -\lla \frac1{\mu(\delta)}\int_\delta\AAA_t\dt\!-\!A_0\rra^2\rrn^\frac12\cdot \label{poboljshana}\\
&\cdot&\!\lln \lln\mathscr{B}\!-\!B_0\rrn^2_\iii\cdot I
-\lla\frac1{\mu(\delta)}\int_\delta\BBB_t\dt-B_0\rra^2\rrn^\frac12\lluo
X\rruo
\nonumber\\
&\le&r_\iii(\AAA)r_\iii(\BBB)\lluo X\rruo,
\nonumber
\end{eqnarray}
and the first half of inequality \eqref{oshtrina0} is proved. The
proof for the remaining  part of  \eqref{oshtrina0} differs from the
previous one only by use of obvious estimates
$$\lln\frac1{\mu(\delta)^2}\int_{\delta^2}\left|\mathscr{A}_s-\mathscr{A}_t\right|^2d(\mu\times\mu)(s,t)\rrn
\le\frac{\diam_\iii^2(\AAA)}2$$ $$\mbox{(resp.
$\lln\frac1{\mu(\delta)^2}\int_{\delta^2}\left|\mathscr{B}_s-\mathscr{B}_t\right|^2d(\mu\times\mu)(s,t)\rrn
\le\frac{\diam_\iii^2(\BBB)}2$)}$$ instead if \eqref{poluprechnik1}
(resp.
           \eqref{poluprechnik2}) in \eqref{poboljshana}.
\end{proof}

Now we turn to the less general case when
$(\mathscr{A}_t)_{t\in\Omega}$ and $(\mathscr{B}_t)_{t\in\Omega}$
are bounded fields of self-adjoint (generally non-commuting)
operators, in which case the above inequality has the following
form.

\begin{corollary}
\label{th1} If $\mu$ is a probability measure on $\Omega$, let
$C,D,E,F$ be bounded self-adjoint operators and let
$(\mathscr{A}_t)_{t\in\Omega}$ and $(\mathscr{B}_t)_{t\in\Omega}$ be
bounded self-adjoint fields satisfying $C\le\mathscr{A}_t\le D$ and
$E\le\mathscr{B}_t\le F$ for all $t\in\Omega$. Then for all
$X\in\ccu$,

\begin{equation}
\llu\int_\Omega\mathscr{A}_tX\mathscr{B}_t\dt-
\int_\Omega\mathscr{A}_t\dt \,X \int_\Omega\mathscr{B}_t \dt\rru
\le\dfrac{\|D-C\|\cdot\|F-E\|}4\cdot\lluo X\rruo. \label{oshtrina}
\end{equation}
\end{corollary}

\begin{proof}
As $\frac{C-D}2\le\mathscr{A}_t-\frac{C+D}2\le\frac{D-C}2$ for every
$t\in\Omega$, then
\begin{eqnarray*}
\supess_{t\in\WWW}\lln \mathscr{A}_t-\frac{C+D}2\rrn&=&
\supess_{t\in\WWW}\sup_{\lln
f\rrn=1}\lla\llp\llm\mathscr{A}_t-\frac{C+D}2 \rrm f,f\rrp\rra\\
&\le& \sup_{\lln f\rrn=1}\lla\llp\frac{D-C}2 f,f\rrp\rra= \frac{\lln
D-C\rrn}2,
\end{eqnarray*}   
which implies
 $r_\iii(\AAA)\le
\frac{\lln D-C\rrn}2,$ and          similarly
$r_\iii(\BBB)\le\frac{\lln F-E\rrn}2.$ Thus \eqref{oshtrina} follows
directly from
    (\ref{oshtrina0}).
\end{proof}

\begin{remark}
Note that similar to the estimate   (\ref{poboljshana}) the
righthand side of \eqref{oshtrina} can be improved to
\begin{eqnarray}
\label{poboljshana1}&&  \lln \lln\frac{D-C}2\rrn^2
   -\lla \int_\Omega\AAA_t\dt-\frac{C+D}2\rra^2\rrn^\frac12\cdot\\
  &\cdot&\lefteqn{  \lln \lln\frac{F-E}2\rrn^2
   -\lla \int_\Omega\BBB_t\dt-\frac{E+F}2\rra^2\rrn^\frac12\lluo X\rruo.}\nonumber
\end{eqnarray}

Estimate similar to (\ref{poboljshana1})  was given in a Gr\"uss
type inequality for  square integrable Hilbert space valued
functions in Theorem 3 in \cite{B-C-D-R}.
\end{remark}

%
 In case of $\HH=\mathbb{C}$  and $\mu$ being the  normalized
Lebesgue measure on $[a,b]$ (i.e. $d\,\mu(t)=\frac{dt}{b-a}$), then
(\ref{grisovaca}) comes as an obvious corollary of  Theorem
\ref{th1}. This special case also confirms the sharpness of the
constant $\frac14$ in the inequality (\ref{oshtrina}).

Taking $\Omega=\{1,\ldots,n\}$ and $\mu$ to be the normalized
counting measure we get  another  corollary of  Theorem \ref{th1},
which gives us the following
 \begin{corollary}[Gr\"uss type inequality for elementary operators]
 Let $A_1,
 \hdots, A_n$, $B_1, \hdots, B_n$, $C, D, E$ and $F$ be bounded linear self-adjoint operators
 acting on a Hilbert space $\HH$
 such that
 $C\le A_i\le D$ and
$E\le B_i\le F$ for all $i=1,2,\cdots,n$ then for arbitrary
$X\in\ccu$,
\begin{eqnarray}\nonumber
\llu  \frac1n\sum_{i=1}^n A_i XB_i-\frac1{n^2}\sum_{i=1}^nA_i\,
X\sum_{i=1}^nB_i\rru \leq \frac{\|D-C\|\|F-E\|}4  \lluo X\rruo\,.
\end{eqnarray}
\end{corollary}



\bibliographystyle{amsplain}

\end{document}